\documentclass{cimart}

\usepackage[shortlabels]{enumitem}
\usepackage{tikz}

\title{Quasi-isometric modification of Gromov-Hausdorff distance}

\authors{Alexei Naianzin}

\authorinfo{V. A. Trapeznikov Institute of Control Sciences of RAS and Lomonosov Moscow State University, Russia}{naianzin.av@phystech.edu}

\abstract{We define a distance analogous to the Gromov–Hausdorff distance that enables the comparison of arbitrary quasi-isometric spaces. We also investigate properties preserved under limits with respect to this distance, as well as properties of the entire class of metric spaces equipped with this distance. For this purpose, we introduce the notion of quasi-isometric distortion for correspondences. Using this notion, we prove that the class of all metric spaces is path-connected; in fact, any two metric spaces can be connected by a curve of finite length.}

\VOLUME{34}
\ISSUE{1}
\NUMBER{8}
\DOI{https://doi.org/10.46298/cm.17523}
\editinfo{February 16, 2026}{May 29, 2026}{Lenny Fukshansky}

\acknowledgments{This work was supported by the Russian Science Foundation under grant no. 25-11-00018. The author is deeply grateful to Andronick Arutyunov for his invaluable guidance, continuous support, and insightful discussions throughout this work. The author also thanks the participants of the seminar on non-commutative geometry and topology and the seminar on the theory of extreme nets for their helpful comments and stimulating conversations. The author is grateful to the anonymous referee for valuable comments and suggestions.}
\begin{document}

\section{Introduction}\label{sec:intro}

The Gromov–Hausdorff distance defines a metric on the set of equivalence classes of compact metric spaces, providing a way to measure the resemblance between metric spaces. 

Recent research has extensively developed the Gromov–Hausdorff framework, including computations of the distance for specific metric spaces and estimates in various general settings \cite{lim2023gromov, adams2025hausdorff, mikhailov2025calculating, mikhailov2025gromov}, the study of the geometry of the space of metric spaces endowed with this distance \cite{borzov2022geometry, bogataya2023clouds}, as well as investigations of limit-preserving properties and compactifications \cite{liu2022gromov, gomez2024gromov}. In addition, numerous modifications of the Gromov–Hausdorff distance have been proposed, often in situations where metric spaces carry additional structure \cite{Khezeli_2020, khezeli2023unified, minguzzi2024lorentzian, bogaty2025fundamentals}. 
Connections between quasi-isometric invariants of groups and their algebraic properties, such as those studied via derivations in group algebras \cite{arutyunov2023derivations}, also motivate further investigation.

While for non-compact metric spaces, the Gromov-Hausdorff distances are often infinite, the convergence of pointed metric spaces doesn't preserve the coarse structure.  So in this paper, we introduce a distance based on the notion of quasi-isometry. By “distance” we mean a symmetric function that does not necessarily satisfy the triangle inequality. The constructed distance allows one to compare arbitrary quasi-isometric spaces. Convergence with respect to this distance will be called \emph{quasi-isometric convergence}.

We show that Gromov–Hausdorff convergence implies quasi-isometric convergence, and that quasi-isometric convergence implies pointed Gromov–Hausdorff convergence, i.e.,
$$ GH \quad \Rightarrow \quad QI \quad \Rightarrow \quad PGH.$$
Precise formulations are provided in Corollary~\ref{cor:GH-->QI} and Proposition~\ref{prop:QI-->PGH}. Consequently, we show that several important properties are preserved under limits with respect to the quasi-isometric distance. These include separability, properness, the intrinsic metric property, and several others; see Corollary~\ref{cor:QI-inherited-properties}.

\subsection{Main results}

Our first main result concerns the behavior of metric properties under quasi-isometric convergence. We show that several natural properties are preserved under limits in the quasi-isometric distance. More precisely, if 
\( X_k \xrightarrow{\rm q.i.} Y \),
then several metric properties of the spaces \(X_k\) are inherited by the limit space \(Y\). These include total boundedness, boundedness of the diameter (with convergence of diameters), separability, properness (under the assumption that \(Y\) is complete), the intrinsic metric property, geodesicity and properness (again under completeness of \(Y\)), as well as \(\delta\)-hyperbolicity and the \( \mathrm{CAT}^\kappa \) condition; see Corollary~\ref{cor:QI-inherited-properties}.

Our second main result shows that the quasi-isometric distance is topologically equivalent to an actual metric. We construct a metric \(D\) (Proposition~\ref{prop:D is a metric}) on the class of metric spaces and prove that it induces the same topology and the same coarse structure as the quasi-isometric distance \(\hat d\). In particular, we prove the following mutual estimates (Proposition~\ref{prop:metric-correspondence}):
\[
\hat d(X,Y)=r \;\Rightarrow\; D(X,Y)\le \ln(1+2r), \qquad
D(X,Y)=r \;\Rightarrow\; \hat d(X,Y)\le e^{2r}-e^{r}.
\]

Finally, we obtain structural results for the space of all metric spaces endowed with the metric $D$. In particular, we prove the following theorem concerning continuous deformations.

{\bf Theorem~\ref{th:Rt-curve}.}{ \it 
Let $X$ and $Y$ be metric spaces, and let $R$ be a correspondence between $X$ and $Y$ with $\operatorname{q-dis}R \le r.$ Then the map $t \mapsto R_t$ is a continuous path in $(\mathfrak{M}, D)$ connecting $X$ and $Y,$ and its length does not exceed $e^{2r} - e^{r}.$ }

Here $\mathfrak{M}$ is the set of equivalence classes of separable metric 
spaces under the equivalence relation $X \sim Y \iff \hat{d}(X, Y) = 0.$ A direct corollary is that the space ($\mathfrak{M}, D$) is linearly connected. Furthermore, the specific construction provided by the theorem shows that its coarse structure is monogenic.

\subsection{Structure of the paper}
In Section~\ref{sec:main-definitions} we recall the basic definitions and formulate the facts needed later concerning Gromov–Hausdorff convergence and convergence of pointed spaces. In Section~\ref{sec:qi-distance} we introduce the quasi-isometric distance, establish some of its properties, and study its relation to ordinary and pointed Gromov–Hausdorff convergence. In particular, we show that several properties are inherited by limits with respect to the quasi-isometric distance. Section~\ref{sec:topol and coarse structure} is devoted to the topology and coarse structure induced by the quasi-isometric distance. In Section~\ref{sec:correspondance-metrization} we introduce the notion of quasi-isometric distortion and construct a generalized metric that induces the same topology and coarse structure as the quasi-isometric distance.

\section{Gromov–Hausdorff distance and basic definitions}
\label{sec:main-definitions}

We first recall the classical definition of the Gromov–Hausdorff distance, list some of its properties and equivalent formulations, and then proceed to the definition of the quasi-isometric distance.

The Hausdorff distance measures how close two subsets of a given metric space are.

\begin{definition}
Let $A$ and $B$ be nonempty subsets of a metric space $X$. The \emph{Hausdorff distance} between $A$ and $B$ is defined by
$$d_H(A, B) = \inf \{r > 0 \mid A \subseteq U_r(B), \, B \subseteq U_r(A)\}, $$
$$U_r(A) = \{x\in X \mid \exists\, a\in A : d(x,a) < r \}. $$
\end{definition}

This notion allows one to define a (possibly infinite) distance between arbitrary metric spaces. The following distance was first introduced by Edwards in 1975 and later independently by Gromov in 1981. A detailed account of its history may be found in \cite{tuzhilin2016invented}. We follow the notation of \cite{burago2001course}.

\begin{definition}
The \emph{Gromov–Hausdorff distance} between nonempty metric spaces $X$ and $Y$ is defined as
\begin{equation} d_{GH}(X, Y)=\inf \left\{d_H\left(X^{\prime}, Y^{\prime}\right) \mid X^{\prime}, Y^{\prime} \subset Z, X^{\prime} \approx X, Y^{\prime} \approx Y\right\} \end{equation}
where the infimum is taken over all metric spaces $Z$ and isometric embeddings of $X$ and $Y$ into $Z$.
\end{definition}

\begin{theorem}
Let $X, Y, Z$ be arbitrary metric spaces. Then
\begin{enumerate}[label=\normalfont\arabic*.]
    \item $d_{GH}(X, Y) = d_{GH}(Y, X);$
    \item $d_{GH}(X, Z) \le d_{GH}(X, Y) + d_{GH}(Y, Z).$
\end{enumerate}
\end{theorem}

Note that reflexivity does not hold, even after passing to the quotient by the relation of isometry.

\begin{example}
Let $\overline{X}$ be the completion of a metric space $X$. Then $d_{GH}(X, \overline{X}) = 0.$
\end{example}

However, on the set of equivalence classes of \emph{compact} metric spaces modulo isometry, the Gromov–Hausdorff distance indeed defines a metric \cite[Theorem 7.3.30]{burago2001course}.

\subsection{Reformulations}
We now present several equivalent ways to compute or estimate the Gromov–Hausdorff distance. For the first of them, we need the notion of a correspondence. Almost all of the definitions below are taken from Sections~7.4 and~8.1 of \cite{burago2001course}.

\begin{definition}
Let $X$ and $Y$ be two sets. A \emph{correspondence} between $X$ and $Y$ is a subset $R \subseteq X \times Y$ satisfying the following condition: for every point $x \in X$ there exists at least one point $y \in Y$ such that $(x,y)\in R$, and similarly, for every point $y \in Y$ there exists a point $x \in X$ such that $(x,y)\in R$. \label{def:distortion}
\end{definition}

A correspondence can be viewed as a surjective multivalued map from $X$ to $Y$, associating to each $x\in X$ the set $\{ y \mid (x,y)\in R \}$.

\begin{definition}
Let $R$ be a correspondence between metric spaces $X$ and $Y$. Its \emph{distortion} $\operatorname{dis} R$ is defined by
\begin{equation}
\operatorname{dis} R=\sup \left\{\left|d_X\left(x, x^{\prime}\right)-d_Y\left(y, y^{\prime}\right)\right|  \mid (x, y),\left(x^{\prime}, y^{\prime}\right) \in R\right\}
\end{equation}
where $d_X$ and $d_Y$ are the metrics on $X$ and $Y$, respectively.

For a map $f\colon X \to Y$ the \emph{distortion} is defined by
\begin{equation}
\operatorname{dis} f = \sup \left\{\left|d_X\left(x, x^{\prime}\right)-d_Y\left(f(x), f(x^{\prime})\right)\right|  \mid x, x' \in X \right\}.
\end{equation}
\end{definition}

\begin{theorem}[\!\!{\cite[Theorem 7.3.25]{burago2001course}}]
For any metric spaces $X$ and $Y$ one has
\begin{equation}
d_{GH}(X, Y)=\frac{1}{2} \inf _{R}(\operatorname{dis} R).
\end{equation}
where the infimum is taken over all correspondences $R$ between $X$ and $Y$.
\end{theorem}

\begin{example} 
A pair of maps $\phi \colon X \to Y$ and $\psi \colon Y \to X$ determines a correspondence 
$R_{\phi \psi} = R_\phi \cup R_\psi$, where $R_\phi, R_\psi \subseteq X\times Y$ are the graphs of the maps $\phi$ and $\psi$.
\end{example}

\begin{proposition}
The following equalities hold:
\[
 \operatorname{dis}  R_{\varphi, \psi}
 =
 \max \left\{
 \operatorname{dis}  R_{\varphi},
 \operatorname{dis}  R_\psi,
 C_{\varphi, \psi}
 \right\},
 \text{ where } 
 C_{\varphi, \psi}
 =
 \sup _{x \in X, y \in Y}
 \left|
 d_X(x, \psi(y)) - d_Y(\varphi(x), y)
 \right|;
\]
\[
    d_{GH}(X, Y)
    =
    \frac{1}{2}
    \inf _{\varphi, \psi}
    \max \left\{
    \operatorname{dis} R_{\varphi},
    \operatorname{dis} R_\psi,
    C_{\varphi, \psi}
    \right\}.
\]
\end{proposition}

This proposition was formulated in \cite{memoli2012some, oles2024efficient}; we omit the straightforward proof. 

One way to obtain estimates for the Gromov–Hausdorff distance is through the notion of an $\varepsilon$-isometry. This concept also allows one to define Gromov–Hausdorff convergence for pointed non-compact metric spaces. 

\begin{definition}
Let $X$ and $Y$ be two metric spaces and let $\varepsilon > 0$. A (possibly discontinuous) map $f\colon X\to Y$ is called an \emph{$\varepsilon$-isometry} if $\operatorname{dis} f \le \varepsilon$ and the image $f(X)$ is an $\varepsilon$-net in $Y$.
\end{definition}

\begin{theorem}[\!\!{\cite[Corollary 7.3.28]{burago2001course}}] \label{th:eps-iso eq gh dist}
Let $X$ and $Y$ be metric spaces and let $\varepsilon>0$. Then:
\begin{enumerate}[\normalfont 1.]
\item if $d_{GH}(X,Y)<\varepsilon$, then there exists a $2\varepsilon$-isometry from $X$ to $Y$;
\item if there exists an $\varepsilon$-isometry from $X$ to $Y$, then $d_{GH}(X,Y)<2\varepsilon$. 
\end{enumerate}
\end{theorem}

\subsection{Properties inherited by limits}
We will write $X_n \xrightarrow{GH} X$ if $d_{GH}(X_n,X)\to 0$. Note that in the case of non-compact spaces the limit is not unique; for example, there exist complete non-isometric metric spaces at zero Gromov–Hausdorff distance, see \cite{ghanaat_gromov_hausdorff}.

The following properties are preserved under Gromov–Hausdorff convergence; see \cite{ghanaat_gromov_hausdorff} and Section 6.1 of \cite{tuzhilin2020lectures}.

\begin{theorem}\label{th:properties of GH limits}
Let \(X_k \xrightarrow{GH} Y\). If each space \(X_k\) has one of the following properties, then the limit space \(Y\) also has this property (assuming $Y$ is complete, when this is required):
\begin{enumerate}[\normalfont 1.]
    \item total boundedness;
    \item boundedness; moreover, \( \operatorname{diam}(X_k)\to \operatorname{diam}(Y) \);
    \item separability;
    \item properness, provided that \(Y\) is complete;
    \item being a length space (assuming \(Y\) is complete);
    \item being a proper geodesic space (assuming \(Y\) is complete);
    \item \( \delta \)-hyperbolicity;
    \item the \( \mathrm{CAT}^\kappa \) condition.
\end{enumerate}
\end{theorem}

For definitions of delta hyperbolicity and \( \mathrm{CAT}^\kappa \) space, see Def. 1.20 on p. 410 and paragraph 1.1 on p. 158 \cite{bridson2013metric}. 

\subsection{The non-compact case}
We now pass to the definition of Gromov–Hausdorff convergence for pointed spaces.

\begin{definition}
A \emph{pointed metric space} is a pair $(X,p)$ consisting of a metric space $X$ together with a fixed base point $p \in X$.\end{definition}

\begin{definition}
    A sequence $\left\{\left(X_n, p_n\right)\right\}_{n=1}^{\infty}$ of pointed metric spaces \emph{converges in the Gromov--Hausdorff sense} to a pointed metric space $(X, p)$ if the following conditions hold: for every $r>0$ and $\varepsilon>0$, there exists a natural number $n_0$ such that for each $n>n_0$, there exists a (not necessarily continuous) map $f: U_r\left(p_n\right) \rightarrow X$ such that:
    \begin{enumerate}[\normalfont 1.]
        \item  $f\left(p_n\right)=p$;
\item $\operatorname{dis} f<\varepsilon$;
\item  the $\varepsilon$-neighbourhood of the set $f\left(U_r\left(p_n\right)\right)$ contains the ball $U_{r-\varepsilon}(p)$. \label{cond: third pointed convergence condition}
\end{enumerate}
We shall use the notation $\left(X_n, p_n\right) \xrightarrow{GH}(X, p)$ for this convergence.
\label{def:punctured-gh-convergence}
\end{definition}

Let $\mathbb S_r^n$ denote the $n$-sphere of radius $r$. The following example of convergence of pointed spaces is typical.

\begin{example}
Let $\{r_k\}$ be a sequence tending to infinity. Then
\[
\left(\mathbb S^n_{r_k},p_k\right) \xrightarrow{GH} (\mathbb R^n,p),
\]
where $p_k$ and $p$ are arbitrarily chosen points on the spheres and on the plane, respectively. \label{ex:punctured-conv}
\end{example}

Note that if in item 3 of Definition~\ref{def:punctured-gh-convergence} one replaces $U_{r-\varepsilon}$ by $U_r$, one obtains a different, non-equivalent definition. The difference appears for spaces with non-intrinsic metrics. For instance, consider
\[
X=\{1\}\cup\{2^m\mid m\in\mathbb N\},\quad p=1,\qquad
X_k=\{1\}\cup\{2^m+\tfrac1k\mid m\in\mathbb N\},\quad p_k=1,
\] we have $\left(X_k,p_k\right)\xrightarrow{GH}(X,p)$, but without subtracting the constant $\varepsilon$ convergence fails.

\begin{proposition}
Properties 3–8 of Theorem~\ref{th:properties of GH limits} remain valid under pointed Gromov-Hausdorff convergence.
\end{proposition}

\begin{proof}
To prove separability of $Y$, it suffices to use the fact that each ball is separable by Theorem~\ref{th:properties of GH limits}. In the remaining cases, it is enough to work inside a sufficiently large ball of finite radius. On finite-radius balls we have ordinary Gromov–Hausdorff convergence, so the properties pass to the limit by Theorem~\ref{th:properties of GH limits}.
\end{proof}

If one additionally assumes that $X_k$ and $Y$ have intrinsic metrics, then the convergence of diameters $\operatorname{diam} X_k \to \operatorname{diam} Y$ also holds. A proof can be found in \cite{jansen2017pointedGH}. However, that paper uses a definition of Gromov–Hausdorff convergence which is not equivalent to ours; but if one restricts to spaces with intrinsic metrics, the definitions become equivalent. Boundedness, of course, is not preserved under taking limits; see Example~\ref{ex:punctured-conv}.

Pointed Gromov–Hausdorff convergence induces a topology on the class of all pointed spaces. It was shown in \cite{Khezeli_2020, noda2024} that this topology is metrizable on the class of proper metric spaces.

\section{Quasi-isometric distance}
\label{sec:qi-distance}

There are several equivalent definitions of quasi-isometry. We will use a definition of quasi-isometry close to that in \cite{bridson2013metric}, see p. 138.

\begin{definition}\label{def: (A, B, C)-quasi-isometry}
Let $(X,d_X)$ and $(Y,d_Y)$ be metric spaces, and let $A \ge 1,B \ge 0, C \ge 0$ be some numbers.
\begin{enumerate}[\normalfont 1.]
\item A map $f \colon X \to Y$ is called an \emph{$(A,B)$-quasi-isometric embedding} if for all $x_1,x_2 \in X$:
\begin{equation}
\frac{1}{A}d_X(x_1,x_2) - B \le d_Y(f(x_1),f(x_2)) \le A d_X(x_1,x_2) + B. \label{eq:qi-emb}
\end{equation}
\item A pair of maps $f \colon X \to Y,$ $g\colon Y \to X$ is called \emph{$C$-inverse} if  there exists $\varepsilon> 0,$ such that for any $x \in X, y\in Y$ the following inclusions hold:
\[
g\left(U_\varepsilon\left(f(x)\right)\right) \subseteq U_C(x), \quad f\left(U_\varepsilon\left(g(y)\right)\right) \subseteq U_C(y).
\]
\item The spaces $X$ and $Y$ are \emph{$(A, B, C)$-quasi-isometric} if there exist $C$-inverse $(A, B)$-quasi-isometric embeddings $f \colon X \to Y$ and $g \colon Y \to X$.
\end{enumerate}
\end{definition}

Spaces $X$ and $Y$ are called \emph{quasi-isometric} if there exist $A, B, C$ such that they are $(A, B, C)$-quasi-isometric.
Using the definition of $(A, B, C)$-quasi-isometry, it is easy to define \emph{quasi-isometric convergence} of spaces.
We write $X_k \xrightarrow{\rm q.i.} Y$ if, starting from some index $N,$ the spaces $X_k$ and $Y$ are $(A_k, B_k, C_k)$-quasi-isometric with $A_k \to 1,$ $B_k \to 0,$ $C_k \to 0.$ 

\begin{remark}
    The above definition of a $C$-inverse map is not standard. This form of the definition is essential for Proposition~\ref{prop: metric is well defined}, where the inverse condition must be applied not at $f(x)$ itself but at a nearby point. The standard requirement is that $d_X(g(f(x)), x) \le \tilde C,$ $d_Y(g(f(y)), y) \le \tilde C$ for any $x \in X, y\in Y.$ Clearly, our definition implies the standard one.  Conversely, if $f, g$ are $(A,B)$-quasi-isometric embeddings that are $\tilde C$-inverse in standard notation, then they are $C$-inverse in our notation for any $C > \tilde C + B$. Proposition~\ref{prop:another-definition of a QI} shows that our definition is indeed equivalent to that given in \cite{bridson2013metric}. 
\end{remark}

In what follows, by a \emph{distance} we mean a symmetric function $d\colon X\times X \to [0, +\infty].$ 
Distances that satisfy all the usual axioms of a metric space and take only finite values will be called \emph{metrics}.
When a metric may take infinite values, we use the term \emph{generalized metric}.
\begin{definition} The \emph{quasi-isometric distance} $\hat d$ between metric spaces $X$ and $Y$ is
$$\hat d(X, Y) = \inf\{ r > 0 \mid X \text{ is $(1+r, r, r)$-quasi-isometric to } Y\}.$$
\end{definition}
Obviously $\hat d (X_k, Y) \to 0$ is equivalent to $X_k \xrightarrow{\rm q.i.} Y.$

The next proposition describes how quasi-isometry constants behave under composition.
\begin{proposition}
    Suppose $X$ is $(A,B,C)$-quasi-isometric to $Y,$ and $Y$ is $(A', B', C')$-quasi-isometric to $Z.$ Then $X$ is $(AA', AB' + A'B, AC' + A' C + B + B')$-quasi-isometric to $Z.$ \label{prop:constants-of-composition}
\end{proposition} 
\begin{proof}
We check by direct computation.
Let $f\colon X\to Y,$ $f' \colon Y\to Z$ be the corresponding quasi-isometric embeddings, then 
\begin{multline*}
d_Z(f'fx_1, f'fx_2) \le A' d_Y(fx_1, fx_2) + B' 
\le A'A d_X(x_1, x_2) + A'B + B' \\
\le A'A d_X(x_1, x_2) + A'B + AB';
\end{multline*}
\begin{multline*}
d_Z(f'fx_1, f'fx_2) \ge \frac{1}{A'} d_Y(fx_1, fx_2) - B' 
\ge \frac{1}{A'A} d_X(x_1, x_2) - \frac{B}{A'} - B' \\
\ge \frac{1}{A'A} d_X(x_1, x_2) - A'B - AB'.
\end{multline*}
Similar inequalities hold for $g, g'.$ 

It remains to check that $f'f$ and $g g'$ are $(AC' + A' C + B + B')$-inverse. Fix $\varepsilon > 0$ small enough so that both the
$C$-inverse condition for $f, g$, and the $C'$-inverse condition for $f', g'$ are satisfied.  For any $x \in X,$ let $z$ be an arbitrary point such that $d_Z(z, f'f(x)) < \varepsilon.$ Then 
\begin{equation}\label{eq: inverse closeness inequality}
d_X(x, gg'z) \le d_X(gg'z, gfx) + d_X(gfx, x) \le A d_Y(g'z, fx) + B + C \le A C' + B +C.
\end{equation}
The analogous estimate for the other order of composition follows similarly. Thus, $X$ is $(A'A, AB' + A'B, AC' + A' C + B + B')$-quasi-isometric to $Z.$ 
\end{proof}
\begin{corollary} The distance $\hat d$ satisfies: $\forall X, Y, Z$
\begin{enumerate}[\normalfont 1.]
    \item $\hat d(X,X) = 0;$
    \item $\hat d(X, Y) = \hat d(Y, X);$
    \item if $\hat d(X, Y) = r,$ $\hat d(Y, Z) = r'$ then $\hat d(X, Z) \le 2(r + r' + rr').$
\end{enumerate} \label{cor:generalised-triangle-equation}
\end{corollary}

\begin{proposition} \label{prop: metric is well defined}
    Let $X,$ $Y_1$ and $Y_2$ be metric spaces. If $\hat{d}(Y_1, Y_2) = 0,$ then \[\hat d(X, Y_1) = \hat d (X, Y_2).\]
\end{proposition}
\begin{proof}
    Fix an arbitrary $r > \hat d(X, Y_1).$ Then there exists an $(1+r, r, r)$-quasi-isometry $\phi\colon X \to Y_1,$ $\psi \colon Y_1 \to X$.
    Let $\alpha \colon Y_1 \to Y_2,$ $\beta \colon Y_2 \to Y_1$ be a 
    $(1+\varepsilon, \varepsilon, \varepsilon)$-quasi-isometry.
    By Proposition~\ref{prop:constants-of-composition}, the compositions 
    $\alpha \phi$ and $\psi \beta$ are $(1+r + \varepsilon(1+r), r + \varepsilon(1+2r))$-quasi-isometric 
    embeddings.

    It remains to show that $\alpha \phi$ and $\psi \beta$ are $(r + O(\varepsilon))$-inverse. 
    Suppose $y_2$ is $\delta$-close to $\alpha \phi(x),$ where $\delta$ is the constant from the $\varepsilon$-inverse condition on $\alpha$ and $\beta$. Then $\beta(y_2)$ is 
    $\varepsilon$-close to $\phi(x),$ and taking $\varepsilon$ small enough, we may assume 
    that $\beta(y_2)$ and $\phi(x)$ are close enough for the $r$-inverse condition 
    on $\phi$ and $\psi$ to be satisfied. Thus, $d_X(\psi\beta(y_2), x) \le r.$

    For the other order, suppose $x$ is $\varepsilon$-close to $\psi \beta(y_2).$ 
    Applying inequality~\eqref{eq: inverse closeness inequality} with 
    $g' = \phi,$ $g = \alpha$ gives 
    $d_{Y_2}(y_2, \alpha \phi(x)) \le (1+\varepsilon)r + 2\varepsilon.$

    Letting $\varepsilon\to 0$ we conclude that $\hat{d}(X, Y_2) \le \hat{d}(X, Y_1).$ The reverse inequality follows analogously. Thus, $\hat{d}(X, Y_2) = \hat{d}(X, Y_1).$
\end{proof}

One can verify that the constructed distance is not a generalized metric, as it fails to satisfy the triangle inequality.
To see this, it suffices to consider spaces consisting of two points with different distances between them.
We say that metric spaces $X$ and $Y$ are equivalent if $\hat d(X, Y) = 0.$
Later, we shall show that using $\hat d$, one can define a topology and a coarse structure on any collection of equivalence classes 
of metric spaces that constitutes a set, and also that there exists a generalized metric $D$ that induces the same topology and coarse structure as $\hat d.$

 If metric spaces $X$ and $Y$ are quasi-isometric, then $\hat d (X, Y) < \infty.$ Thus, we obtain a finite distance on the class of quasi-isometric 
spaces. However, unlike pointed Gromov--Hausdorff convergence, convergence with respect to our distance still tracks the coarse structure of the space, in the sense that $\hat d(X_k, X) \to 0$ implies that from some index onward, the spaces $X_k$ and $Y$ are coarsely equivalent.
It suffices to choose an index beyond which all values $\hat d(X_k, X)$ are finite.

The next example shows that $\hat d$ allows for the comparison of a wider class of spaces than the Gromov--Hausdorff distance.
\begin{example}
 Let $X$ be the real line with the standard metric, and $Y$ a polygonal chain whose $k$-th segment has length $k$ and all angles are right angles; see Figure~\ref{fig:placeholder}.
The metric on $Y$ is induced from the plane. Then $$d_{GH}(X, Y) = \infty, \quad \hat d(X, Y) < \infty.$$ One can use Theorem~\ref{th:eps-iso eq gh dist} to see this. 
\begin{figure}
     \centering
    \begin{tikzpicture}[every node/.style={transform shape}, scale=0.6]
    \tikzstyle{every path}=[thick]
    \tikzstyle{dot}=[fill, circle, inner sep=1.5pt]
    \draw (0, 3) -- (12, 3);
    \node[left] at (0, 3) { $X$:};
    \node [below] at (2,3) {$\Tilde{a}$};
\node [above] at (5,3) {$\Tilde{b}$};
    \node [below] at (11,3) {$\Tilde{c}$};
    \node[dot] at (2,3) {};
    \node[dot] at (5,3) {};
\node[dot] at (11,3) {};
    \draw (0, 0) -- (1, 1) -- (3, -1) -- (6, 2) -- (10, -2) -- (13, 1);
\foreach \i in {1,...,4} {
         \node[dot] at ({13 + 0.3*\i}, {1 + 0.3*\i}) {};
}
    \node[left] at (0, 0) { $Y$:};
    \node [below] at (3,-1) {$a$};
    \node [above] at (6,2) {$b$};
\node [below] at (9,-1) {$c$};
    \node[dot] at (3,-1) {};
    \node[dot] at (6,2) {};
    \node[dot] at (9,-1) {};
\end{tikzpicture}
     \caption{Spaces that are  infinitely far in Gromov-Hausdorff distance but not in $\hat d$  }
     \label{fig:placeholder}
 \end{figure}
\end{example}

It is well known that the existence of a quasi-isometric embedding with dense image implies quasi-isometry of the spaces.
However, we are interested in concrete estimates on the constants; therefore, we prove the following proposition.

\begin{proposition} \label{prop:another-definition of a QI}
    Let $f\colon X \to Y$ be an $(A, B)$-quasi-isometric embedding whose image is $R$-dense.
Then the spaces $X$ and $Y$ are $(A, AB + 2R, 3A(B+R))$-quasi-isometric.
\end{proposition}
\begin{proof}
    Define $g \colon Y \to X.$ For each point $y \in Y$, there exists a (possibly not unique) point $x \in X$ such that $d_Y(y, f(x)) < R.$ Set $g(y) = x.$ For the map $g$, we have the inequalities
    \begin{align*} d_X(g(y_1), g(y_2)) \le A d_Y(f(g(y_1)), f(g(y_2))) + AB \le A d_Y(y_1, y_2) + AB + 2R;
\\
    d_X(g(y_1), g(y_2)) \ge \frac{1}{A} d_Y(f(g(y_1)), f(g(y_2))) - B \ge \frac{1}{A} d_Y(y_1, y_2) - B - \frac{2R}{A}.
\end{align*}
    Hence, $g$ is an $(A, B+2R)$-quasi-isometry.
Now we need to obtain a constant for which these maps are inverse. Fix $\varepsilon > 0$ such that $A\varepsilon < B.$ Let $x$ be a point $\varepsilon$-close to $g(y).$ Then 
$$
        d_Y(f(x), y) \le d_Y(f(x), fg(y)) + d_Y(fg(y), y) \le A\varepsilon + B + R \le 3A(B+R).$$

        For the other order of compositions, let $y$ be a point $\varepsilon$-close to $f(x).$ Then
        \[d_X(g(y), x) \le d_X(g(y), gf(x)) + d_X(gf(x), x) \le 3A(B+R),\]
        where the last inequality is obtained as follows:
        \begin{align*}
        &d_X(g(y), gf(x)) \le  A \varepsilon + AB + 2R \le A(2B + 2R) ,\\
        &d_X(gf(x), x) \le A d_Y(fgf(x), f(x)) + AB \le A(R  +B).
        \end{align*}
That completes the proof.
\end{proof}
\begin{corollary}
    Let $X,$ $Y$ be metric spaces.
Then $\hat d (X, Y) \le 9\, d_{GH} (X, Y).$
    \label{cor:GH-->QI}
\end{corollary}
\begin{proof}
    Indeed, let $d_{GH} (X, Y) = r,$ then there exists a $2 r$-isometry $f$ from $X$ into $Y.$ Then $f$ is a $(1, 2r)$-isometry with $r$-dense image.
Hence, $X$ and $Y$ are $(1, 4 r, 9 r)$-quasi-isometric. Therefore, $\hat d (X, Y) \le 9 r.$ 
\end{proof}

Thus, Gromov--Hausdorff convergence implies quasi-isometric convergence.
The converse, of course, is false. The example below also appears in \cite{mikhailov2025newgeodesiclinesgromovhausdorff}. 

\begin{example}
    Let $\alpha_k$ be a sequence of irrational numbers such that $\alpha_k \to 1$ as $k \to \infty.$
    Consider $X_k = \alpha_k \mathbb Z.$ Then $\hat{d}(X_k, X) \le |1-\alpha_k|
\to 0,$ 
    but $d_{GH}(\alpha_k \mathbb Z , \mathbb Z) \ge \frac{1}{4}.$ The latter inequality follows from Theorem~\ref{th:eps-iso eq gh dist} and the fact that the distortion of any map $\alpha_k \mathbb Z \to \mathbb Z$ is at least $\frac{1}{2}.$ 
\end{example}

We show that quasi-isometric convergence implies pointed Gromov--Hausdorff convergence.
\begin{proposition}
    If $X_k \xrightarrow{\rm q.i.} Y,$ then there exist $p_k\in X_k$ and $p \in Y,$ such that \[(X_k, p_k) \xrightarrow{GH} (Y, p).\] That is, quasi-isometric convergence implies pointed convergence.
\label{prop:QI-->PGH}
\end{proposition} 
\begin{proof}
    Since $\hat d (X_k, Y) \to 0,$ there exist maps $g_k \colon X_k \to Y,$ such that $A_k \to 1, B_k\to 0,$ $R_k\to 0$ where $A_k,$ $B_k$ are the constants from the quasi-isometric embedding formula \eqref{eq:qi-emb} for $g_k,$ and each $g_k$ is $R_k$-dense. Choose $p$ arbitrarily, and select $p_k$ such that \[d_Y(g_k(p_k), p) < R_k.\]
    Set $f_k (x) = g_k(x)$ for $x \neq p_k,$ $f_k(p_k) = p.$ Thus, $f_k$ satisfies condition 1. 
    Because $A_k \to 1,$ $B_k \to 0$, on the ball of radius $r$ we have $\operatorname{dis} \left.{f_k}\right|_{B_r(p_k)} \to 0.$ Condition \hyperref[cond: third pointed convergence condition]{3} holds because $f_k$ is $2R_k$-dense and $R_k \to 0.$
\end{proof}
\begin{corollary}
    Let \( X_k \xrightarrow{\rm q.i.} Y \).
If each space \( X_k \) possesses one of the following properties, then the limit space \( Y \) also possesses that property (provided \( Y \) is complete, if indicated):
\begin{enumerate}[\normalfont 1.]
    \item total boundedness;
\item boundedness of diameter; moreover, \( \operatorname{diam}(X_k) \to \operatorname{diam}(Y) \);
    \item separability;
\item properness of the space, if \( Y \) is complete;
\item the property of being a length space, if \( Y \) is complete;
\item properness and geodesicity of the space, if \( Y \) is complete;
    \item  \( \delta \)-hyperbolicity;
\item the \( \mathrm{CAT}^\kappa \) condition.
\end{enumerate}
That is, the metric $\hat d$ satisfies conditions 2–9 of Theorem~\ref{th:properties of GH limits}.
\label{cor:QI-inherited-properties}
\end{corollary}
\begin{proof}
    Properties 3–8 hold because they hold for pointed convergence.
As noted earlier, under quasi-isometric convergence, the coarse structure is preserved.
Therefore, if $\operatorname{diam} X = \infty,$ then from some index onward $\operatorname{diam} X_k = \infty.$ If $\operatorname{diam} X < \infty,$ then from some index onward the diameters $X_k$ are finite, and on uniformly bounded spaces, quasi-isometric convergence is equivalent to Gromov--Hausdorff convergence.
Hence, items 1 and 2 remain valid. 
 \end{proof}

Thus, convergence with respect to $\hat d$ is a weaker convergence than Gromov--Hausdorff convergence, but it still preserves many important properties, i.e., it generalizes the Gromov--Hausdorff distance.

\section{Topology and coarse structure}
\label{sec:topol and coarse structure}

We now proceed to describe the topology and coarse structure induced by our distance.
We work with the set $\mathfrak{M}$ of equivalence classes of separable metric spaces under the relation of being at zero quasi-isometric distance. Note that by Proposition~\ref{prop: metric is well defined} the metric $\hat d$ is well-defined on equivalence classes. 

\subsection{Coarse structure} First, we recall the definition of a coarse structure.
\begin{definition}
A \emph{coarse structure} on a set $X$ is a family of subsets $\mathcal{E}$ of the Cartesian product $X \times X$ satisfying the following axioms:
\begin{enumerate}[\normalfont 1.]
\item The diagonal $\Delta_X = \{(x,x)\mid x\in X\}$ belongs to $\mathcal{E}.$
\item Symmetry:
If $E \in \mathcal{E}$, then $E^{t}=\{(y, x) \mid(x, y) \in E\}$ also belongs to $\mathcal{E}$.
\item Closure under finite unions: If $E_1, E_2, \ldots, E_n \in \mathcal{E}$, then $\bigcup_{i=1}^n E_i \in \mathcal{E}$.
\item Closure under taking subsets:
If $E \in \mathcal{E}$ and $F \subseteq E$, then $F \in \mathcal{E}$.
\item Closure under composition:
For any $E, F \in \mathcal{E}$ we have  $E \circ F \in \mathcal{E}$, where \begin{equation} E \circ F=\{(x, z) \mid \exists\, y \in X,(x, y) \in E \text{ and }  (y, z) \in F\}.\end{equation}
\end{enumerate}
\end{definition}

We provide simple examples of coarse structures.
\begin{example}
 The \emph{maximal coarse structure} consists of all subsets of $X\times X.$
    \end{example}
    \begin{example}
 The \emph{minimal coarse structure} consists of all subsets of $\Delta_X.$
    \end{example}
    \begin{example}
    Let $(X, d)$ be a metric space.
The system $\mathcal{E}_d$ consisting of sets $E\subseteq X\times X$ such that 
    \begin{equation} \exists \, r > 0 : E \subseteq E_r, \quad \textnormal{where} \quad E_r = \{(x, x') \mid d(x,x') \le r \}.
\label{eq:metric-coarse-structure} \end{equation}
    is a coarse structure. It is called the \emph{bounded metric coarse structure}.
\end{example}

Recall that $\mathfrak{M}$ is the set of equivalence classes of separable metric spaces.
By analogy with the metric coarse structure, we define the family  $\mathcal{E}_{\hat d}$.
That is, $E \in \mathcal{E}_{\hat d}$ if and only if it satisfies condition \eqref{eq:metric-coarse-structure} for $\hat d$.

\begin{proposition}   
The family $\mathcal{E}_{\hat d}$ defines a coarse structure on $\mathfrak{M}$.
\end{proposition}
\begin{proof}
    Conditions 1–4 are trivially satisfied. By Corollary~\ref{cor:generalised-triangle-equation}
    $$E_{r_1} \circ E_{r_2} \subseteq E_{2(r_1 + r_2 + r_1r_2)}, $$
    hence property 5 also holds.
\end{proof}

We now make a few remarks about the metrizability of coarse structures.
\begin{definition} A coarse structure $\mathcal{E}$ on a set $X$ is called \emph{metrizable} if there exists a generalized metric $d$ such that $\mathcal{E} = \mathcal{E}_d.$ 
\end{definition}

\begin{definition} Let $\mathcal{S} \subseteq 2^{X \times X}.$ The coarse structure \emph{generated} by $\mathcal{S}$ is the coarse structure obtained as the intersection of all coarse structures containing $\mathcal{S}.$
\end{definition}

\begin{theorem}[\!\!{\cite[Theorem 2.55]{Roe-en}}]
    A coarse structure is metrizable if and only if it is generated by some countable family of sets $\mathcal{S}.$
\end{theorem}

\begin{example}
    The coarse structure $\mathcal{E}_{\hat{d}}$ on $\mathfrak{M}$ is countably generated by the sets 
    $$E_n = \{(X, Y) \mid \hat{d}(X, Y) < n \}, $$ therefore, there exists a generalized metric that induces this coarse structure.
\end{example}

It is easy to give an explicit formula for a generalized metric that induces the same coarse structure as $\hat{d}.$

\begin{proposition}
    The distance $\rho = 2 + \ln( \hat d+1)$ for $\hat d \neq 0,$ and zero otherwise, is a generalized metric and induces the same coarse structure as $\hat d.$
\end{proposition}
\begin{proof}
    Let $\hat{d}(X, Y) = r,$ $\hat{d} (Y, Z) = r'.$ If $r$ or $r'$ equals 0, then the other two distances coincide by Proposition~\ref{prop: metric is well defined}. Otherwise, by Corollary~\ref{cor:generalised-triangle-equation} 
    $$\rho(X, Z) \le \ln(2(r + r' + rr') +1) + 2 \le \ln (r + 1)(r' +1) + 4 = \rho(X, Y) + \rho(Y, Z).$$
  Thus, $\rho$ is a generalized metric. That they induce the same coarse structure is obvious.
\end{proof}
However, the constructed metric induces the discrete topology; thus, by passing to such a metric, we lose information about the convergence of spaces.

\subsection{Topology} We turn to the description of the topology on the set $\mathfrak{M}.$

\begin{definition}[Topology $\mathcal{T}_{\hat{d}}$]  A set $\mathfrak{U} \subseteq \mathfrak{M}$ is called \emph{open} if for each $X \in \mathfrak{U}$ and every sequence \( X_k \xrightarrow{\rm q.i.} X \), there exists $n_0$ such that $X_k \in \mathfrak{U}$ for $k\geq n_0$.
\end{definition}

\begin{proposition} A base for this topology is formed by the sets 
$$\displaystyle B_r(X) = \{ Y \mid \hat d(X, Y) < r \}.$$
\end{proposition}
\begin{proof}  Obviously for each $r > 0$ and $X \in \mathfrak{M}$, the set $B_r(X)$ lies in $\mathcal{T}_{\hat d}.$ Let $\mathfrak{U}$ be an open subset of $\mathfrak{M},$ $X \in \mathfrak{U}.$ Then there exists $r > 0$ such that $B_r(X) \subseteq \mathfrak{U}.$ Indeed, otherwise one could find a sequence $X_k \in B_{\frac{1}{k}}(X)\backslash \mathfrak{U}$ with $X_k \to X.$ This contradicts the openness of $\mathfrak{U}.$
\end{proof}

\begin{remark}
    It is easy to see that for the topology and coarse structure on $\mathfrak{M}$ it is not important how $\hat d$ is exactly defined as a function of the constants $A, B, C$ from Definition~\ref{def: (A, B, C)-quasi-isometry}. What is important is that $X_k \to X$ if and only if $A_k \to 1,$ $B_k \to 0,$ $C_k \to 0$ as $k \to \infty, $ and $E \in \mathcal{E}_{\hat d}$ if there exist $A, B, C$ large enough such that for any pair $(X, Y) \in E,$ the spaces $X$ and $Y$ are $(A,B,C)$-quasi-isometric. Thus, the topology and coarse structure are naturally defined.  
\end{remark}

From the results of \cite{Chittenden-1917-en}, \cite{frink}, it follows that the topology on $\mathfrak{M}$ 
is metrizable. Also, with a minor modification, the results of \cite{Wright2011} apply here. We will not elaborate on this now, because in the next section we will present an explicit construction of a generalized metric that induces the same topology and coarse structure as $\hat d.$ 

Thus, we have defined on $\mathfrak{M}$ a topology and a coarse structure and have verified that each of them separately is metrizable.
We now proceed to prove that they are simultaneously metrizable.

\section{Correspondences and metrization}
\label{sec:correspondance-metrization} We now show how to define a metric equivalent to the quasi-isometric distance using the notion of correspondence (see Definition~\ref{def:distortion}). To simplify notation, we denote the distance between points \( x \) and \( x' \) by \( |x, x'| \). The metric space in which the distance is taken will be clear from the context or stated explicitly.
\begin{definition}
    The \emph{quasi-isometric distortion} of a correspondence $R$ is the number 
    \begin{multline}
        \operatorname{q-dis}R = \inf \bigg\{ r> 0 \, \bigg| \, \frac{1}{e^r}|y,y'|
- e^r + 1 \le|x,x'| \le e^r|y,y'| + e^{2r} - e^r \\  \text{ for all } (x,y), (x',y') \in R \bigg\}.
\end{multline}
\end{definition}

\begin{definition}
    Define the distance from $X$ to $Y$ by 
    \begin{equation}
        D(X, Y) = \inf_R\{\operatorname{q-dis} R\},
    \end{equation}
    where the infimum is taken over all correspondences $R$ between $X$ and $Y.$ 
\end{definition}
\begin{proposition}
    The constructed distance is a metric.
\label{prop:D is a metric}
\end{proposition}
\begin{proof}
    Reflexivity and symmetry are trivial. Let us check the triangle inequality.
Suppose $D(X, Y) = r_0,$ $D(Y, Z) = s_0.$ Then there exists a correspondence $R$ between $X$ and $Y$ with $\operatorname{q-dis} R = r < r_0 + \varepsilon$ and a correspondence $S$ between $Y$ and $Z$ with $\operatorname{q-dis} S = s < s_0 + \varepsilon.$ Define a correspondence between $X$ and $Z$ by 
    \[
        T = \{(x, z) \mid \exists \, y : (x, y) \in R, (y, z) \in S \}.
\]
    Now we estimate its quasi-isometric distortion. 
    \[
        |x,x'|
\le e^{r}|y,y'| + e^{2r} - e^r \le e^{r+s}|z,z'| + e^{r}(e^{2s} - e^{s}) + e^{2r} - e^r
    \]
    It is easy to see that 
    \[
        e^{r}(e^{2s} - e^{s}) + e^{2r} - e^r \le e^{2r + 2s} - e^{r + s}.
\]
    Indeed, by rearranging and grouping terms, we see that this is equivalent to the inequality
    \[
        e^{2r} - e^r \le e^{2s} (e^{2r} - e^r).
\]
    The reverse inequalities are established analogously.
Thus, $\operatorname{q-dis} T \le r + s,$ and because $\varepsilon$ is arbitrary, we get $D(X, Z) \le r_0 + s_0.$
\end{proof}

We now show that $\operatorname{q-dis}R$ induces the same topology and coarse structure as the quasi-isometric distance.
\begin{proposition} Let $X,$ $Y$ be metric spaces. \begin{enumerate}[\normalfont 1.]
\item If $\hat d(X, Y) = r,$ then $D(X, Y) \le \ln(1 + 2r);$ 
\item If $D(X, Y) = r,$ then $\hat d(X, Y) \le e^{2r} - e^r.$
Thus, the topologies and coarse structures induced by $\hat d$ and $D$ coincide.
\end{enumerate} \label{prop:metric-correspondence}
\end{proposition}
\begin{proof}
Indeed, suppose there is a correspondence $R$ between the spaces $X$ and $Y$ with parameters $(A, B).$ That is, for any pairs $(x,y), (x', y') \in R$, we have 
$$\frac{|y,y'|}{A} - \frac{B}{A} \le |x,x'|
\le A|y,y'| + B.$$
 By definition, for a correspondence, we have $\pi_X R = X,$ $\pi_Y R =  Y,$ where $\pi_X, \pi_Y$ are the projections from the Cartesian product $X \times Y$ onto the corresponding factor.
Hence, there exist maps $f\colon X \to Y,$ $g\colon Y \to X,$ such that $\Gamma_f,$ $\Gamma_g^T \subseteq R.$ Since $R$ is an $(A, B)$-correspondence, $f,$ $g$ are $(A,B)$-quasi-isometric embeddings.
Let us check that $f$ and $g$ are $C$-inverse for any $C > B$. Fix $\varepsilon > 0$ such that $A\varepsilon + B < C$.  If $y$ is $\varepsilon$-close to $f(x),$ then 
$$ |gy,x| \le A |fx, y| + B \le C.  $$ 
Therefore, $X$ and $Y$ are $(A, B, C)$-quasi-isometric.
The estimate for the second composition is completely analogous. Thus, from $D(X, Y) = r$, it follows that
 \[
 \hat d (X, Y) \le \max\{e^r - 1, e^{2r} - e^r\} = e^{2r} - e^r.
\]

 Conversely, let $X$ and $Y$ be $(A, B, C)$-quasi-isometric spaces.
That is, there exist $(A, B)$-quasi-isometric embeddings $f\colon X\to Y,$ $g \colon Y \to X$ and $f, g$ are $C$-inverse.
Define 
 \[ R = \Gamma_f \cup \Gamma_g^T \]
 and estimate $\operatorname{q-dis}R.$ Let $(x, y), (x', y') \in R.$ If  $(x, y), (x', y')$ lie in the graph of one of the functions, then 
 \[ \frac{1}{A}|x,x'|
- \frac{B}{A} \le |y,y'| \le A|x,x'| + B,  
 \]
 because the functions are $(A, B)$-quasi-isometric embeddings.
If $(x, y) \in \Gamma_f,$ $(x', y') \in \Gamma_g^T,$ then 
 \[ |x,x'| = |x, g(y')| \le |gfx, gy'|
+ |x, gfx| = |gy, gy'| + C \le A|y,y'| + B + C. \]
The estimate for $|yy'|$ is done similarly.
Thus, we have found an $(A, B + C)$-correspondence between $X$ and $Y.$ Hence, if $\hat d(X, Y) = r,$ then 
 \[ e^{D(X, Y)} \le 1 + r, \quad e^{2D(X, Y)} - e^{D(X, Y)} \le 2r, \]
 whence $D(X, Y) \le \ln(1 + 2r).$

 Therefore, both distances induce the same topology and coarse structure.
\end{proof}

\subsection{Continuous deformations of metric spaces}
We now show that any two quasi-isometric metric spaces can be connected by a continuous path of finite length.
Let $R$ be a correspondence between metric spaces $X$ and $Y.$ Denote the set $R$ equipped with the metric 
\begin{equation}|(x,y), (x',y')|_t = (1-t)|xx'|
+ t|yy'|
\end{equation}
by $R_t.$  For $t = 0$ and $t = 1$, one generally obtains pseudometric spaces.
But we identify them with the corresponding metric spaces obtained by identifying points at distance zero.
Then $ X = R_0$, $Y = R_1.$ This construction of a one-parameter family was proposed in \cite{Ivanov_2016}.

In the next theorem we obtain a bound on the length of curves, constructed by correspondences. We don't claim that the obtained bound is tight. 
\begin{theorem} Let $X$ and $Y$ be metric spaces, and let $R$ be a correspondence between $X$ and $Y$ with $\operatorname{q-dis}R \le r.$ Then the map $t \mapsto R_t$ is a continuous path in $(\mathfrak{M}, D)$ connecting $X$ and $Y,$ and its length does not exceed $e^{2r} - e^{r}.$ 
\label{th:Rt-curve}
\end{theorem}
\begin{proof}
First, we show continuity.
Let $t, s \in [0,1]$, $|t - s| = \delta$.
The graph of the identity map $\operatorname{id} \colon R_t \to R_s$ gives a correspondence between $R_t$ and $R_s.$ We estimate its distortion.
\begin{multline*}
|(x,y), (x',y')|_s = (1-s)|x,x'| + s|y,y'| = (1-t)|x,x'| + t|y,y'|
+ (s-t) (|y,y'| - |x,x'|) \le \\ \le |(x,y), (x',y')|_t + \delta ||x,x'| - |y,y'||
\end{multline*}
We estimate the term with $\delta$ separately; without loss of generality assume $|x,x'| > |y,y'|.$
\begin{align*}
 |x,x'| - |y,y'| = t|x,x'|
+ (1-t)|x,x'| - t|y,y'| - (1-t)|y,y'| \le \\
\le t e^r |y,y'| + t(e^{2r} - e^r) + (1-t)|x,x'|
-   t|y,y'| - \frac{(1-t)}{e^r}|xx'| + (1-t)(e^r - 1)  =\\
= t(e^r - 1)|y,y'| + (1-t)(1 - \frac{1}{e^r})|x,x'|
+ (1-t) (e^{r}-1) + t(e^{2r} - e^r) \le\\
(e^r -1) |(x,y), (x', y')|_t + (e^{2r} - e^r)
\end{align*}
Thus, 
$$|(x,y), (x',y')|_s \le (1 + \delta (e^r - 1))|(x,y), (x',y')|_t + \delta(e^{2r} - e^r).$$
Hence as $\delta \to 0,$ the distortion of $\operatorname{id}\colon R_t \to R_{t+\delta}$ tends to $0,$ so the curve is continuous.
Now, we write asymptotic estimates that will help us compute the length of the curve.
Let $\Delta = \operatorname{max}(\Delta_1, \Delta_2),$ where $\Delta_1, \Delta_2$ are determined from the equations
$$e^{\Delta_1} = 1 + \delta (e^r - 1), \quad e^{2\Delta_2} - e^{\Delta_2} = \delta(e^{2r} - e^r).$$
Clearly $D(R_t, R_s) \le \Delta.$ As $\delta \to 0,$ we have 
$$\Delta_1 = \delta(e^{r} - 1) + o(\delta), \quad \Delta_2 = \delta(e^{2r} - e^r) + o(\delta).$$
Therefore,
$\Delta =  \delta(e^{2r} - e^r) + o(\delta).$
Using the inequality $D(R_t, R_s) \le \Delta$ and passing to sufficiently fine partitions of the curve, we obtain that the length of the curve does not exceed $e^{2r} - e^{r}.$ 
\end{proof}

\begin{corollary}
The set of equivalence classes $\mathfrak{M}$ of separable metric 
spaces is monogenic.  
\end{corollary}
\begin{proof}
The coarse structure is generated by the set $E_1 = \{(X, X') \mid D(X, X') \le 1\}.$ Indeed, from the fact that any two metric spaces at distance at most $r$ can be connected by a curve of length at most $e^{2r} - e^r,$ it follows that 
$E_r \subseteq E_1 \circ \dots \circ E_1,$
where the right‑hand side is the composition of $k$ sets with $k > e^{2r} - e^r.$
\end{proof}

The above reasoning allows one to pass to an intrinsic metric, i.e., a metric defined as the infimum of lengths of curves between points.
From the estimates on curve lengths, it follows that this generalized metric will induce the same topology and coarse structure as $D,$ and hence as $\hat d.$

\begin{example} Let $E_p = (\mathbb R^n, \|\cdot \|_p)$ denote the $n$-dimensional normed space with the norm
$$ \|x \|_p = \sqrt[p]{|x_1|^p + \dots + |x_n|^p }. $$
Consider the curve $\gamma(p) = E_p,$ where $1 \le p \le \infty.$ We will show that this curve is a geodesic (though not naturally parametrized) on the interval $[1,2]$ and on the ray $[2, +\infty)$. Note that for $n = 2$ their union forms a closed loop (since $E_1$ is isometric to $E_\infty$), hence the union will not be a geodesic.

Recall that the Banach--Mazur distance between isomorphic Banach spaces $X$ and $Y$ is defined by
\begin{equation} D_{BM} (X, Y) = \inf_{T} \ln \max\{\|T\|, \| T^{-1}\| \},\end{equation}
where the infimum is taken over all isomorphisms $T \colon X \to Y.$ 

We now explain why the distance $D(X, Y)$ between finite-dimensional vector spaces $X$ and $Y$ of the same dimension equals the Banach--Mazur distance. First, observe that the space $\lambda X := (X, \lambda \|\cdot \|_X)$ is isometric to $X.$ Consequently, $X$ coincides with its asymptotic cone $\operatorname{ascone} X$. Moreover, to every $(A, B, C)$-quasi-isometry $\tilde f \colon X \to Y$ corresponds a bi-Lipschitz homeomorphism $ f \colon \operatorname{ascone} X \to \operatorname{ascone} Y,$ with Lipschitz constants $\operatorname{L} (f)$ and $\operatorname{L}(f^{-1})$ not exceeding $A$ \cite[Lemma 10.83]{dructu2018geometric}.  Since $f$ and $f^{-1}$ are Lipschitz, they are differentiable almost everywhere, there exists a point $x_0$ of differentiability s.t. the derivative $df_{x_0}$ is a linear isomorphism between $X$ and $Y$ with $\operatorname{L}(df_{x_0}) \le \operatorname{L}(f) \le A.$ Thus, from each $(A, B, C)$-quasi-isometry we obtain a linear isomorphism with norm at most $A.$ In Proposition~\ref{prop:metric-correspondence} we showed that every $(A, B)$-correspondence gives rise to a pair of quasi-isometries with the same parameter $A,$ which yields the required equality.    

It is known \cite[Proposition 37.6]{tomczak1989banach} that for $1 \le p \le q \le 2,$ and for $2 \le p \le q \le \infty,$ 
\begin{equation} D_{BM}(E_p, E_q) = \left(\frac{1}{p} - \frac{1}{q}\right) |\ln n|,\end{equation}
and this value is attained by the identity map. 

Considering a broken line $E_{p_1}, \dots, E_{p_k}$ with $p_1 < p_2<\dots < p_k,$ we see that 
$$D(E_{p_1}, E_{p_2}) + \dots + D(E_{p_k-1}, E_{p_k}) = \left(\frac{1}{p_1} - \frac{1}{p_k}\right)| \ln n| = D(E_{p_1}, E_{p_k}). $$
Hence, the supremum of the lengths of broken lines equals the distance, so the curve is a geodesic. 

Now consider the curve $t \mapsto R_t,$ where $R_t$ denotes the set $\mathbb R^n$ endowed with the norm defined by $(1-t)\|x \|_1 + t \| x \|_2.$ 
Fix $s < t$; then 
$$D_{BM}(R_t, R_s) \le \ln  \frac{(1-t) + tM}{(1-s) + sM},$$ 
where $\ln M = D_{BM}(R_0, R_1) = D_{BM}(E_1, E_2).$ 
Observe that 
\begin{align*} D_{BM}(E_1, E_2) &\le D_{BM}(R_{t_0}, R_{t_1}) + \dots + D_{BM}(R_{t_{n-1}}, R_{t_{n}}) \\ 
&= \ln \left( \dfrac{1-t_1 + t_1 M}{1} \cdot \dfrac{1-t_2 + t_2 M}{1-t_1 + t_1 M}\cdot \dots \cdot \dfrac{M}{1-t_{n-1} + t_{n-1}M} \right) \\
&= \ln M = D_{BM}(E_1, E_2). \end{align*}
Therefore, the curve $t \mapsto R_t$ is also a geodesic.
\end{example}

Thus, the estimates obtained in the previous section are not optimal in general.

\end{document}